\documentclass[11pt,reqno]{amsart}
\usepackage{amsmath,amsthm,amscd,amsfonts,amssymb,color}
\usepackage{cite}
\usepackage[mathscr]{eucal}
\usepackage[bookmarksnumbered,colorlinks,plainpages]{hyperref}
\newtheorem{theorem}{Theorem}[section]

\newtheorem*{Acknowledgement}{\textnormal{\textbf{Acknowledgement}}}

\theoremstyle{definition}
\newtheorem{definition}[theorem]{Definition}

\newtheorem{Open Prob}[theorem]{Open Problem}
\theoremstyle{remark}

\numberwithin{equation}{section}
\def\DJ{\leavevmode\setbox0=\hbox{D}\kern0pt\rlap{\kern.04em\raise.188\ht0\hbox{-}}D}
\begin{document}

\title[A short proof of the metrizability of $\mathcal{F}$-metric spaces]{A short proof of the metrizability of $\mathcal{F}$-metric spaces}

\author[S.\ Som, L.K.\ Dey,]
{Sumit Som$^{1}$, Lakshmi Kanta Dey$^{2}$}

\address{{$^{1}$} Sumit Som,
                    Department of Mathematics,
                    National Institute of Technology
                    Durgapur, India.}
                    \email{somkakdwip@gmail.com}
\address{{$^{2}$} Lakshmi Kanta Dey,
                    Department of Mathematics,
                    National Institute of Technology
                    Durgapur, India.}
                    \email{lakshmikdey@yahoo.co.in}


\begin{abstract}
The main purpose of this manuscript is to provide a short proof of the metrizability of $\mathcal{F}$-metric spaces introduced by Jleli and Samet in \cite[\, Jleli, M. and Samet, B., On a new generalization of metric spaces, J. Fixed Point Theory Appl. (2018) 20:128]{JS1}.

\end{abstract}

\maketitle

\setcounter{page}{1}

\section{\bf Introduction}
\baselineskip .55 cm
Recently, Jleli and Samet \cite{JS1} proposed a new generalization of  usual metric space concept. By means of a certain class of functions, the authors defined the notion of an $\mathcal{F}$-metric space. Firstly, we will recall the definition of such kind of spaces. Let $\mathcal{F}$ denote the class of functions $f:(0,\infty)\rightarrow \mathbb{R}$ which satisfy the following conditions :

($\mathcal{F}_1$) $f$ is non-decreasing, i.e., $ 0<s<t\Rightarrow f(s)\leq f(t)$.

($\mathcal{F}_2$) For every sequence $\{t_n\}_{n\in \mathbb{N}}\subseteq (0,+\infty)$, we have
$$\lim_{n\to +\infty}t_n=0 \Longleftrightarrow \lim_{n\to +\infty}f(t_n)=-\infty.$$


Now due to Jleli and Samet, the definition of an $\mathcal{F}$-metric space is as follows:

\begin{definition} \cite{JS1} \label{D1}
Let $X$ be a non-empty set and $D:X\times X\rightarrow [0,\infty)$ be a given mapping. Suppose there exists $(f,\alpha)\in \mathcal{F}\times [0,\infty)$ such that
\begin{enumerate}
\item[(D1)] $D(x,y)=0\Longleftrightarrow x=y~\forall~(x,y)\in X \times X$.
\item[(D2)] $D(x,y)=D(y,x),~$ $\forall~ (x,y)\in X \times X$.
\item[(D3)] For every $(x,y)\in X\times X$, for each $N\in \mathbb{N}, N\geq2$ and for every $(u_i)_{i=1}^{N}\subseteq X $ with $(u_1,u_N)=(x,y)$, we have
$$D(x,y)>0 \Longrightarrow f(D(x,y))\leq f\left(\sum_{i=1}^{N-1}D(u_i,u_{i+1})\right)+ \alpha.$$
\end{enumerate}
Then $D$ is said to be an $\mathcal{F}$-metric on $X$ and the pair $(X,D)$ is said to be an $\mathcal{F}$-metric space.
\end{definition}
In our earlier manuscript \cite{SAL}, we have already proved that this new generalization of metric space is indeed metrizable by a metric $d:X\times X\rightarrow \mathbb{R}$ defined as follows 
\begin{equation}
d(x,y)=\mbox{inf}\left\{\sum_{i=1}^{N-1}D(u_i, u_{i+1}): N\in \mathbb{N}, N\geq 2, \{u_i\}_{i=1}^{N}\subseteq X~\mbox{with}~(u_1,u_N)=(x,y)\right\}.\label{e}
\end{equation}
In that proof, we showed that $\tau=\tau_{\mathcal{F}}$ where $\tau$ denotes the topology generated by $d$ in \ref{e} and $\tau_{\mathcal{F}}$ denotes the topology generated by the $\mathcal{F}$-metric $D.$ We also showed that the notions of Cauchy sequence, completeness, Banach contraction principle are equivalent to that of usual metric spaces. In this manuscript, we will give a very short proof of the metrizability of $\mathcal{F}$-metric spaces. We use Chittenden's metrization theorem \cite{CD} in our proof. Further, for more details one can see \cite{FR}. 

Before proceeding to our main result, we first recall the metrization result due to Chittenden \cite{CD}. Let $X$ be a topological space and $F:X \times X\rightarrow [0,\infty)$ be a distance function on $X$. If the distance function $F$ satisfies the following conditions:
\begin{enumerate}
\item[(i)] $F(x,y)=0\Longleftrightarrow x=y~\forall~(x,y)\in X \times X$.

\item[(ii)] $F(x,y)=F(y,x),~$ $\forall~ (x,y)\in X \times X$.

\item[(iii)](Uniformly regular) For every $\varepsilon>0$ and $x,y,z \in X$ there exists $\phi(\varepsilon)>0$ such that if $F(x,y)<\phi(\varepsilon)$ and $F(y,z)<\phi(\varepsilon)$ then $F(x,z)<\varepsilon$
\end{enumerate}
then the topological space $X$ is metrizable. 


\section{\bf Main Result}
In this section, we will provide a short proof of the metrizability of $\mathcal{F}$-metric spaces.

\begin{theorem} \label{FMS}
Let $X$ be an $\mathcal{F}$-metric space with $(f,\alpha)\in \mathcal{F}\times [0,\infty)$. Then $X$ is metrizable.
\end{theorem}

\begin{proof}
Let $X$ be an $\mathcal{F}$-metric space with $(f,\alpha)\in \mathcal{F}\times [0,\infty)$. By the definition of an $\mathcal{F}$-metric space, the distance function $D:X \times X\rightarrow [0,\infty)$ satisfies the first two conditions of Chittenden's metrization result i.e,
\begin{enumerate}
\item[(i)] $D(x,y)=0\Longleftrightarrow x=y~\forall~(x,y)\in X \times X$;

\item[(ii)] $D(x,y)=D(y,x),~$ $\forall~ (x,y)\in X \times X$.
\end{enumerate}
Now we will prove the third condition i.e., the `uniformly regular' condition. Let $\varepsilon>0$ and $x,y,z \in X.$ If $x=z$, then $D(x,z)=0$. So in this case $\phi(\varepsilon)=c$ where $c$ is any positive real number will serve the purpose. Now let $x\neq z.$ Then $D(x,z)>0.$ So by the definition of an $\mathcal{F}$-metric space we have,
\begin{equation}
f(D(x,z))\leq f(D(x,y)+D(y,z))+\alpha.
\label{aa}
\end{equation}
Now by the $\mathcal{F}_2$ condition, for $(f(\varepsilon)-\alpha)\in \mathbb{R}$ there exists $\delta>0$ such that $0<t<\delta \Rightarrow f(t)<f(\varepsilon)-\alpha.$ Now let us choose $\phi(\varepsilon)=\frac{\delta}{2}.$ If $D(x,y)< \frac{\delta}{2}$ and $D(y,z)<\frac{\delta}{2}$ then $D(x,y)+D(y,z)<\delta.$ So by the equation \ref{aa}, we have 
$$ f(D(x,z))<f(\varepsilon)$$
$$\Rightarrow D(x,z)<\varepsilon.$$
This shows that the distance function $D$ of an $\mathcal{F}$-metric space satisfies the uniformly regular condition. Consequently,  by Chittenden's metrization result we can conclude that the $\mathcal{F}$-metric space $X$ is metrizable.
\end{proof}

\begin{Acknowledgement}
 The Research is funded by the Council of Scientific and Industrial Research (CSIR), Government of India under the Grant Number: $25(0285)/18/EMR-II$. 
\end{Acknowledgement}

\bibliographystyle{plain}

\end{document}